\documentclass[12pt]{article}
\usepackage{amsmath,amsfonts,amssymb}
\usepackage{alltt}

\usepackage{amsmath,amsfonts,ifthen,fullpage,enumerate}  
\usepackage{mathrsfs}
\usepackage{hyperref}
\usepackage{mathtools}

\usepackage{enumerate}

\usepackage{amsmath,amsfonts,amssymb}

\usepackage{makecell}  

\usepackage{graphicx}        

\usepackage[table]{xcolor}   

\usepackage{tcolorbox}

\usepackage[all,pdf]{xy}
\usepackage{lmodern,amssymb}   

\def\Co{\mathop{\rm Co}\nolimits}

\def\spam{\mathop{\rm span}\nolimits}

\def\Re{\mathop{\rm Re}\nolimits}
\def\Im{\mathop{\rm Im}\nolimits}

\def\Gram{\mathop{\rm Gram}\nolimits}

\def\Sym{\mathop{\rm Sym}\nolimits}
\def\trace{\mathop{\rm trace}\nolimits}

\def\rank{\mathop{\rm rank}\nolimits}

\def\diag{\mathop{\rm diag}\nolimits}

\def\pmat#1{\begin{pmatrix}#1\end{pmatrix}}
\def\mat#1{\begin{matrix}#1\end{matrix}}
\def\question#1{{\bf Question: }#1}
\def\question#1{}

\def\cG{{\cal G}}

\def\cO{{\cal O}}
\def\cT{{\cal T}}

\def\R{\mathbb{R}}

\def\CC{\mathbb{C}}

\def\FF{\mathbb{F}}
\def\HH{\mathbb{H}}

\def\OO{\mathbb{O}}
\def\PP{\mathbb{P}}
\def\QQ{\mathbb{Q}}
\def\ZZ{\mathbb{Z}}

\def\Cd{\C^d}
\def\Fd{\FF^d}
\def\Od{\OO^d}
\def\Hd{\HH^d}
\def\Hn{\HH^n}
\def\Rd{\R^d}

\def\C{\mathbb{C}}

\def\implies{\Longrightarrow}
\def\SS{\mathbb{S}}

\newcommand{\RR}{\mathbb{R}}

\newtheorem{theorem}{Theorem}[section]

\newtheorem{lemma}{Lemma}[section]
\newtheorem{example}{Example}[section]

\newtheorem{proposition}{Proposition}[section]
\newtheorem{definition}{Definition}[section]
\newtheorem{conjecture}{Conjecture}
\newenvironment{proof}{{\noindent \it
Proof.}}{\hfill$\Box$\medskip}
\newif\ifdraft\def\draft{\drafttrue\hoffset=.8truecm\showlabeltrue
\def\comment##1{{\bf comment: ##1}}
\headline={\sevenrm \hfill \ifx\filenamed\undefined\jobname\else\filenamed\fi%
(.tex) (as of \ifx\updated\undefined???\else\updated\fi)
 \TeX'ed at {\hour\time\divide\hour by 60{}%
\minutes\hour\multiply\minutes by 60{}%
\advance\time by -\minutes
\the\hour:\ifnum\time<10{}0\fi\the\time\  on \today\hfill}}
}

\def\inpro#1{\langle#1\rangle}
\def\ip#1{\langle\kern-.28em\langle#1\rangle\kern-.28em\rangle_\nu}

\def\tensor{\otimes}
\def\cC{{\cal C}}
\def\cD{{\cal D}}
\def\cO{{\cal O}}
\def\cT{{\cal T}}
\def\cI{{\cal I}}
\def\cH{{\cal H}}

\def\cV{{\cal V}}
\def\cW{{\cal W}}
\def\V{\cV}

\def\norm#1{\Vert#1\Vert}
\def\openR{{{\rm I}\kern-.16em {\rm R}}}

\def\Fd{\FF^d}
\let\ga\alpha

\let\gb\beta

\let\gG\Gamma
\let\gd\delta
\let\gD\Delta

\let\gth\theta

\let\gl\lambda

\let\gL\Lambda

\let\gs\sigma

\let\go\omega
\let\gO\Omega
\let\ga\alpha

\let\gam\gamma
\let\gG\Gamma
\let\gb\beta
\let\gd\delta
\let\gs\sigma
\def\inpro#1{\langle#1\rangle}

\def\Im{\mathop{\rm Im}\nolimits}

\def\ker{\mathop{\rm ker}\nolimits}
\def\GL{\mathop{\it GL}\nolimits}

\def\Iff{\hskip1em\Longleftrightarrow\hskip1em}
\def\Implies{\hskip1em\Longrightarrow\hskip1em}

\def\formeq{\the\sectionno.\the\equationno}  
\def\elabel#1/#2/#3/{\global\advance\equationno by 1 %
\ifx#1\empty\else\emember#1%
\ifshowlabel\marginal{\string#1}\fi\fi%
\ifmmode\eqno{#3(\formeq#2)}\else#3\formeq#2\fi} 

\def\makeblanksquare#1#2{
\dimen0=#1pt\advance\dimen0 by -#2pt
      \vrule height#1pt width#2pt depth0pt\kern-#2pt
      \vrule height#1pt width#1pt depth-\dimen0 \kern-#1pt
      \vrule height#2pt width#1pt depth0pt \kern-#2pt
      \vrule height#1pt width#2pt depth0pt
}

\title{\bf 
The geometry of the six quaternionic equiangular lines in $\HH^2$
}

\author{Shayne Waldron\\ }

\begin{document}

\maketitle 

\begin{abstract}
We give a simple presentation of the six quaternionic equiangular lines in $\HH^2$
as an orbit of the primitive quaternionic reflection group of order $720$
(which is isomorphic to $2\cdot A_6$, the double cover of $A_6$).
Other orbits of this group are also seen to give optimal spherical designs (packings)
 of $10$, $15$ and $20$ lines in $\HH^2$,
	with angles $\{ {1\over3},{2\over3} \}$,  $\{ {1\over4},{5\over8}\}$ 
	and $\{0,{1\over3},{2\over3} \}$, respectively.
We consider the origins of this reflection group as one of Blichfeldt's ``finite collineation groups'' for lines in $\CC^4$,
and general methods for finding nice systems of quaternionic lines.


\end{abstract}

\bigskip
\vfill


 \ifdraft (Other designs, configurations) \else\fi
15B33, \ifdraft (Matrices over special rings (quaternions, finite fields, etc.) \else\fi
20C25, \ifdraft Projective representations and multipliers \else\fi
20G20, \ifdraft Linear algebraic groups over the reals, the complexes, the quaternions \else\fi
51M05, \ifdraft (Euclidean geometries (general) and generalizations) \else\fi
51M20, \ifdraft (Polyhedra and polytopes; regular figures, division of spaces [See also 51F15]) \else\fi
\quad
secondary
15B57, \ifdraft (Hermitian, skew-Hermitian, and related matrices) \else\fi
51E99, \ifdraft Geometry,  None of the above, but in this section \else\fi
51M15, \ifdraft (Geometric constructions in real or complex geometry) \else\fi
65D30. \ifdraft (Numerical integration) \else\fi

\vskip .5 truecm
\hrule
\newpage

\section{Introduction}

There has been considerable  interest in determining maximal sets of equiangular lines
in the Euclidean spaces $\RR^d$, $\Cd$ and $\Hd$, as part of the theory of spherical 
designs, since it began in the 1970's \cite{DGS77}.
The existence of real equiangular lines corresponds to the existence of 
certain classes of strongly regular graphs \cite{W09}, the existence of a maximal 
set of $d^2$ equiangular lines in $\Cd$, i.e., a SIC, is conjectured by Zauner to 
hold for all dimensions \cite{Zau10}, \cite{ACFW18}, and for quaternionic lines 
(see \cite{CKM16}) no explicit examples of maximal sets of quaternionic equiangular lines were known 
until Et-Taoui \cite{ET20} presented the following (maximal) set of six 
equiangular lines in $\HH^2$ given by the unit vectors
\begin{align}
\label{ETlines}
&v_1=\pmat{ 1\cr 0 }, \
v_2=\pmat{ {\sqrt{2}\over\sqrt{5}} \cr {\sqrt{3}\over\sqrt{5}} }, \
v_3=\pmat{ {\sqrt{2}\over\sqrt{5}} \cr -{\sqrt{3}\over4\sqrt{5}}+{3\over4}i }, \
v_4=\pmat{ {\sqrt{2}\over\sqrt{5}} \cr -{\sqrt{3}\over4\sqrt{5}}-{1\over4}i+{1\over\sqrt{2}}j }, \cr
&v_5=\pmat{ {\sqrt{2}\over\sqrt{5}} \cr -{\sqrt{3}\over4\sqrt{5}}-{1\over4}i-{1\over2\sqrt{2}}j+{\sqrt{3}\over2\sqrt{2}}k }, \
v_6=\pmat{ {\sqrt{2}\over\sqrt{5}} \cr
-{\sqrt{3}\over4\sqrt{5}}-{1\over4}i-{1\over2\sqrt{2}}j-{\sqrt{3}\over2\sqrt{2}}k },
\end{align}
%
which are said to have ``projective symmetry group'' $A_6$.
They are unique up to projective unitary equivalence, 
and were obtained by solving the system of polynomials giving the equiangularity, i.e.,
$$ |\inpro{v_j,v_k}|^2 = {2\over5}, \quad j\ne k, \qquad |\inpro{v_j,v_j}|^2=1,  $$
in the variables $z_a,w_a\in\CC^2$, where $v_a=z_a+w_a j\in\HH^2$, and 
\begin{equation}
\label{inprodefn}
\inpro{v,w} := \sum_{j=1}^d \overline{v_j}w_j \quad
\hbox{(Euclidean inner product)}.
\end{equation}

In this paper, we give a simple presentation of the six equiangular lines in $\HH^2$,
in an attempt to understand quaternionic equiangular lines in general
(in the context of maximal sets of real and complex equiangular lines).
The main presentation roughly follows our investigation, showing the motivation, 
definitions required, and technical details, such as computations undertaken 
in Magma. We summarise the key findings:

\begin{itemize}
\item The six equiangular lines are seen to be the orbit of an (irreducible) 
quaternionic reflection group $H_{720}$ of order $720$, 
which is isomorphic to $2\cdot A_6$, the double cover of $A_6$. 
For the purpose of comparison with (\ref{ETlines}), the corresponding 
presentation of the lines is given by the equal-norm vectors
\begin{align}
\label{nicelines}
& \pmat{\sqrt{2}+\sqrt{10}\cr \sqrt{3}-i+j+\sqrt{3}k}, \
 \pmat{\sqrt{2}+\sqrt{10}\cr 2i-2j}, \
 \pmat{\sqrt{2}+\sqrt{10}\cr -\sqrt{3}-i+j-\sqrt{3}k}, \cr
& \pmat{-\sqrt{3}-i-j-\sqrt{3}k \cr \sqrt{2}+\sqrt{10}}, \
\pmat{2i+2j \cr \sqrt{2}+\sqrt{10}}, \
\pmat{\sqrt{3}-i-j+\sqrt{3}k \cr \sqrt{2}+\sqrt{10}}.
\end{align}
\item The lines are also the orbit (of a vector in $\HH^2$) under the action of 
a subgroup of $H_{720}$ of order $24$, which is the complex reflection subgroup with Shephard-Todd 
number $4$.
Interestingly, the orbit of another (complex-valued) vector under this subgroup 
gives the four equiangular lines in $\CC^2$.
\item The equiangular lines are not roots of the reflection group $H_{720}$, which partly explains 
why they have only recently been found.
Indeed, the stabiliser of a line is a reflection free subgroup of order $120$ which has
a faithful action on the line.
\item The projective action of $H_{720}\cong 2\cdot A_6$ 
(which has centre the scalar matrices $\{-1,1\}$) 
on the six lines is that of $A_6$, i.e., there are exactly two elements $\pm g$ 
of the reflection group which give any even permutation of the lines.
\item The quaternionic reflection groups have been classified, and $H_{720}$ is Cohen's
group of type $O_2$ \cite{C80} (page 320). This is said to be Blichfeldt's collineation
group (C) for $\CC^4$ \cite{B17} (page 142). 
We consider Blichfeldt's groups (A), (C) and (K) in detail.
These are neither presented as reflection groups nor as collineation groups for quaternionic lines,
but are seen to give nested (irreducible) quaternionic reflection groups which permute various
finite sets of quaternionic lines.
\item In principal, it is possible to go from the abstract group $2\cdot A_6$ to its 
rank $2$ quaternionic representation $H_{720}$, and then the six quaternionic 
equiangular lines in $\HH^2$ (each one of which is fixed by a reducible
subgroup of order $120={720\over 6}$). 
This provides a general method for constructing 
nice (highly symmetric) sets of quaternionic lines from abstract groups.
\end{itemize}

We assume some 
familiarity with the 
quaternions $\HH$, which are a noncommutative division algebra, and the linear
algebra over them \cite{C80}, \cite{Z97}, \cite{W20}, \cite{V21}.
This can be routine, e.g., matrix groups (don't swap the order of multiplication), 
to extremely involved, e.g., defining the determinant (it can't reasonably be done \cite{A96}).
We will provide appropriate commentary when required.

Throughout, we adopt the following conventions. The Euclidean quaternionic space $\Hd$ will 
be thought of as a right $\HH$-vector space (module), so that $\HH$-linear maps are applied on 
left, and we have
$$ A(v\ga) = (Av)\ga, \qquad A\in M_d(\HH), \ v\in\Hd, \ \ga\in\HH, $$
and the inner product (\ref{inprodefn}) satisfies
$$ \inpro{v\ga,w \gb} = \overline{\ga}\inpro{v,w}\gb, \qquad \ga,\gb\in\HH, \ v,w\in\Hd. $$
We use the ``complexification''
\begin{align}
\label{symplecticform}
	g=A+Bj\in M_d(\HH) & \Iff [g]_\CC:=\pmat{A&-B\cr \overline{B}&\overline{A}}\in M_{2d}(\CC), \cr
	v=z+wj\in\Hd & \Iff [v]_\CC:=\pmat{z\cr\overline{w}}\in\CC^{2d},
\end{align}
where $[\cdot]_\CC$ is $\CC$-linear, $[gv]_\CC= [g]_\CC[v]_\CC$, etc.

\section{The primitive quaternionic reflection group} 

The projective symmetry group (see \cite{CW18}, \cite{W18}) of the six equiangular lines
is $A_6$, which is generated by the permutations
\begin{equation}
\label{abpermdefs}
a=(1 2)(3 4) \quad\hbox{(order $2$)}, \qquad
b=(1 2 3 5)(4 6) \quad\hbox{(order $4$)}.
\end{equation}
In \cite{W20}, it was shown that unitary matrices 
which give these permutations
of the six equiangular lines (\ref{ETlines}) are given by 
\begin{align}
\label{UaUbdefn}
U_a &= \pmat{
{2\over\sqrt{15}}i-{\sqrt{2}\over\sqrt{15}}j & {\sqrt{2}\over\sqrt{5}}i -{1\over\sqrt{5}}j \cr
{\sqrt{2}\over\sqrt{5}}i -{1\over\sqrt{5}}j
& -{2\over\sqrt{15}}i+{\sqrt{2}\over\sqrt{15}}j }, \cr
U_b &= \pmat{
{1\over2\sqrt{5}}+{1\over2\sqrt{3}}i+{3-\sqrt{5}\over2\sqrt{30}}j+{\sqrt{5}+1\over2\sqrt{10}}k
& {\sqrt{3}\over2\sqrt{10}}-{1\over2\sqrt{2}}i+{3+\sqrt{5}\over4\sqrt{5}}j
-{\sqrt{3}\over5+\sqrt{5}} k \cr
{\sqrt{3}\over2\sqrt{10}}+{1\over2\sqrt{2}}i +{3-\sqrt{5}\over4\sqrt{5}}j
+{\sqrt{3}\over5-\sqrt{5}}k &
-{1\over2\sqrt{5}}+{1\over2\sqrt{3}}i-{3\sqrt{5}+5\over10\sqrt{6}}j
+{\sqrt{5}-1\over2\sqrt{10}}k }.
\end{align}
These satisfy
$$ U_a^2=-I, \qquad U_b^4=-I, $$
and so they do not generate $A_6$, as a matrix group.

The six equiangular lines of (\ref{ETlines}) and the unitary matrices of (\ref{UaUbdefn})
can be put in the {\tt Magma} computer algebra system 
as a quaternion algebra over the field $\QQ(\zeta)$, 
$\zeta=\zeta_{120}=e^{2\pi i\over 120}$, 
which contains the required roots $\sqrt{2},\sqrt{3},\sqrt{5}$, e.g., 
\begin{verbatim}
   F:=CyclotomicField(24);  
   zeta:=RootOfUnity(24);  
   Q<i,j,k>:=QuaternionAlgebra<F|-1,-1>;
\end{verbatim}
Since $i=\zeta^{30}\in \QQ(\zeta)$, 
care must be taken to distinguish (or identify) it with the
$i$ provided by {\tt QuaternionAlgebra}. We did this by entering all quaternions
in the form $a_1+a_2 i+a_3 j +a_4 k$, where $a_j\in\RR\cap \QQ(\zeta)$ and
$i,j,k$ are the units for the quaternion algebra.
Since Magma could form, but not calculate with, the matrix group
(this feature has subsequently been added)
\begin{equation}
\label{Hdef}
H:=\langle U_a,U_b\rangle\subset U_2(\HH),
\end{equation}
it was constructed as a group of complex matrices via (\ref{symplecticform}), 
from which it was deduced that $H$ is the finite group of order $720$ with small 
group identifier {\tt <720, 409>}. This group is $2\cdot A_6$, the double cover of $A_6$,
which is the unique (abstract) group with 
$$ \hbox{composition series:} \qquad
\mat{2\cdot A_6 & \cr | & A_6 \cr * & \cr | & \ZZ_2 \cr  1 & } $$
In other words:
\begin{itemize}
\item $H$ is a rank $2$ faithful irreducible unitary representation of $2\cdot A_6$ over $\HH$, which could be obtained from the complex representations of the abstract
	group $2\cdot A_6$.
\end{itemize}
A {\bf quaternionic reflection} is defined to be a nonidentity unitary map $g\in U_d(\HH)$
which fixes a subspace of dimension $d-1$ of $\Hd$ (for us these will have finite order,
and sometimes this is taken as part the definition), i.e.,
$$ \rank(I-g)=d-1. $$
A nonzero vector in the orthogonal complement of the fixed subspace of a reflection
(which defines the fixed subspace) is called a {\bf root} of the reflection,
and the subspace it gives a {\bf root line}.
A (usually finite) group generated by quaternionic reflections is called a 
{\bf quaternionic reflection group} \cite{C80} 
or a {\bf symplectic reflection group} \cite{BST23} (when given as a complex matrix group).
The conjugate $h^{-1}gh$ in the unitary group $U_d(\HH)$, or $M_d(\HH)$, 
of a reflection $g$
is a reflection, since
$$ \rank(I-h^{-1}gh) = \rank(h^{-1}(I-g)h) = \rank(I-g). $$
Correspondingly, the reflection groups are classified up to conjugation in $U_d(\HH)$.
By directly observing that $H$ contains reflections, and then calculating the 
reflection group generated by these reflections, it was determined that
\begin{itemize}
\item $H$ is a primitive quaternionic reflection group with $40$ reflections, each of order $3$.
\item The $40$ reflections correspond to $20$ root lines (a reflection and its inverse have
	the same root line), which is the maximum possible.
\item $H$ is generated by three reflections, but not two.
\item The centre of $H$ is the scalar matrices $\pm I$.
\end{itemize}
Since $U_a$ and $U_b$ permute the six equiangular lines, so does $H$, with
$\pm g$ giving the same even permutation of the lines. In this way, elements of
$H$ can be indexed by the permutation of the six equiangular lines that they give
(they cover this element of $A_6$).
$$ \mat{ \multicolumn{3}{c}{\hbox{\bf Conjugacy classes of $A_6$}} \\[0.2cm]
\hbox{order}&\hbox{representative}&\hbox{length}\cr
1 & () & 1 \cr
2 & (12)(34) & 45 \cr
3 & (123)(456) & 40 \cr
3 & (123) & 40 \cr
4 & (1234)(56) & 90 \cr
5 & (12345) & 72 \cr
5 & (13452) & 72
} 
\qquad \qquad
\mat{ \multicolumn{4}{c}{\hbox{\bf Conjugacy classes of $2\cdot A_6$}} \\[0.2cm]
\hbox{order}& \hbox{length} & 
\hbox{elements covered} 
& \hbox{lines fixed} \cr
1 & 1 & () & 6 \cr
2 & 1 & () & 6 \cr 
\ 3^* & 40 & (123)(456) 
& \hbox{none} \cr
3 & 40 & (123) & 3 \cr
4 & 90 & (12)(34) & 2 \cr
5 & 72 & (12345) & 1 \cr
5 & 72 & (13452) & 1 \cr
6 & 40 & (123)(456) & \hbox{none} \cr
6 & 40 & (123) & 3 \cr
8 & 90 & (1234)(56) & \hbox{none} \cr
8 & 90 & (1234)(56) & \hbox{none} \cr
10 & 72 & (12345) & 1  \cr
10 & 72 & (13452) & 1  \cr
\multicolumn{4}{l}{\hbox{\footnotesize * The $40$ reflections of order $3$}} 
} $$
In particular, we observe that
\begin{itemize}
\item The $40$ reflections form a  conjugacy class of elements of order three, corresponding to $(123)(456)$,
i.e., they fix no equiangular lines. 
Thus the equiangular lines are not roots of the reflections, 
nor of their orthogonal complements.
\item The $40$ elements of order three which are not reflections form a conjugacy class,
	corresponding to $(123)$, i.e., they fix  
		three equiangular lines and cycle the remaining three. 
\end{itemize}


The six equiangular lines of (\ref{ETlines}) are the orbit of $v_1=e_1$ under the action of
the reflection group $H$. Therefore, the first column of each matrix in $H$ is a vector in
one of the lines.
Any unitary image of these lines is a set of six equiangular lines, which is an orbit
of the corresponding conjugate of the group $H$ in $U_2(\HH)$.
We now consider such a presentation of the lines given by a conjugate of $H$,
for which a generating set of reflections takes a simple form.

\section{The Blichfeldt generators}

The (irreducible) quaternionic reflection groups were classified in Cohen \cite{C80}.
Cohen classifies the primitive quaternionic reflection groups by whether 
their complexification is primitive or not (also see \cite{S23}). 
For our group $H$, the complexifications of the generators $U_a$ and $U_b$ 
are elements of order $4$ and $8$ with eigenvalues $i,i,-i,-i$ 
and $\pm\sqrt{i},\pm i\sqrt{i}$. 
It is easily verified that the $1$-dimensional eigenspaces for 
$[U_b]_\CC$ have trivial intersection with the $2$-dimensional 
ones for $[U_a]_\CC$, and so the complexification of $H$ is primitive. 
Cohen gives a list of $16$ ``exceptional'' primitive quaternionic reflection groups
with primitive complexifications, in dimensions $1,2,3,4,5$, of which there are six 
for $\HH^2$, including a unique one of order $720$. Therefore

\begin{itemize}
\item $H$ is the unique primitive quaternionic reflection group of order $720$.
\end{itemize}

Cohen \cite{C80} describes the primitive quaternionic reflection group of order $720$ as Blichfeldt's
``primitive simple group of collineations in four variables (C) of order $360\phi$''.
The generators for Blichfeldt's group (C) given in \cite{B17} (page 141) are, in the 
notation of the day:
$$ F_1=(1,1,\go,\go^2), \quad \go=-{1\over2}+{\sqrt{3}\over 2}i, $$
$$ F_2: \ x_1={1\over\sqrt{3}}(x_1'+\sqrt{2}x_4'),\ 
           x_2={1\over\sqrt{3}}(-x_2'+\sqrt{2}x_3'),\
           x_3={1\over\sqrt{3}}(\sqrt{2}x_2'+x_3'),\
           x_4={1\over\sqrt{3}}(\sqrt{2}x_1'-x_4'), $$
$$ F_3: \ x_1={1\over2}(\sqrt{3}x_1'+x_2'), \ 
	  x_2={1\over2}(x_1'-\sqrt{3}x_2'), \
	  x_3=x_4', \ 
	  x_4=x_3', $$
$$ F_4: \ x_1=x_2', \ 
          x_2=x_1', \
	  x_3=-x_4', \
	  x_4=-x_3', $$
corresponding to the ``substitutions of the alternating group'' (permutations)
$(abc)$, $(ab)(cd)$, $(ab)(de)$, $(ab)(ef)$, respectively. 
It is one of the $30$ types of ``primitive collineation groups in four complex variables'',
and is simple. 
Effectively, it is a projective representation of the alternating group $A_6$ 
(a simple group of order $360$) 
given as a group of matrices (to be factored by the order $\phi$ subgroup 
of scalar matrices). 
Since, as the name suggests, collineation groups map lines to lines (here in $\CC^4$),
I initially thought the permutation action was on six lines in $\CC^4$, but in fact 
the action is conjugation on a self normalising index $6$ subgroup. 
We have now seen -- over a hundred years later --
that it also corresponds to a permutation action on six quaternionic lines.

In modern matrix notation, the above generators are
$$ \pmat{1&0&0&0\cr 0&1&0&0\cr 0&0&\go&0\cr 0&0&0&\go^2},\ 
{1\over\sqrt{3}}\pmat{1&0&0&\sqrt{2}\cr
                      0&-1&\sqrt{2}&0\cr
                      0&\sqrt{2}&1&0\cr
                      \sqrt{2}&0&0&-1}, \
\pmat{ {\sqrt{3}\over2} & {1\over2} & 0 & 0 \cr
       {1\over2} & -{\sqrt{3}\over2} & 0 & 0 \cr
       0 & 0 & 0 & 1 \cr
       0 & 0 & 1 & 0},\
\pmat{0&1&0&0 \cr
      1&0&0&0 \cr
      0&0&0&-1 \cr
      0&0&-1&0}. $$
These all have determinant $1$, and generate a group $G$ of order $1440$, 
with centre the scalar matrices $\langle i\rangle$, i.e., $\phi=4$, and $G/Z(G) \cong A_6$.
There is a related subgroup (A) of order $60\phi$ and a supergroup (K) of order $720\phi$ 
given by generators:
$$ (A):\ F_1,F_2,F_3, \qquad (C):\ F_1,F_2,F_3,F_4, \qquad (K):\  F_1,F_2,F_3,F_4,F'', $$
where $F''$ is given by 
$$ F'': \ x_1=x_2', \
          x_2=-x_1', \
          x_3=x_4', \
          x_4=-x_3', \qquad
\pmat{0&1&0&0 \cr -1&0&0&0 \cr 0&0&0&1 \cr 0&0&-1&0}. $$
These both have centre $\langle i\rangle$, with $G/Z(G)$ being $A_5$ and $S_6$, respectively.
Blichfeldt's matrices are not in the form (\ref{symplecticform}) for symplectic matrices. 
The first can be put
in this form by conjugation with the permutation matrix 
$P=[e_1,e_3,e_2,e_4]$ of order $2$, i.e.,
$$ P F_1 P = \pmat{1&0&0&0 \cr 0&\go&0&0 \cr 0&0&1&0 \cr 0&0&0&\overline{\go}},
\quad PF_2 P ={1\over\sqrt{3}} \pmat{
        -1 & 0 & 0 & -\sqrt{2} \cr
        0 & -1 & -\sqrt{2} & 0 \cr
        0 & -\sqrt{2} & 1 & 0 \cr
        -\sqrt{2} & 0 & 0 & 1 }, 
$$
but the second does not have the form (\ref{symplecticform}). 
By multiplying $F_2,F_3,F_4$ by $\pm i$, 
which does not change the determinant, we obtain matrices which $P$ conjugates to the desired form. 
The corresponding collineation groups, thus obtained, have centre $\langle -1\rangle$, 
and can be viewed as subgroups of $U_2(\HH)$. 
Consider the matrices $a_1,\ldots,a_5$ so obtained from
$$ PF_1^2P, \quad P(iF_2)P, \quad P(-iF_3)P, \quad  P(-iF_4)P,\quad P(-F'')P, $$
i.e.,  
\begin{equation}
\label{Blichaform}
a_1=\pmat{1&0\cr0&\go^2}, \
a_2=\pmat{ {1\over\sqrt{3}}i & -{\sqrt{2}\over\sqrt{3}}k \cr
 -{\sqrt{2}\over\sqrt{3}}k & {1\over\sqrt{3}}i }, \
a_3=\pmat{{1\over2}k -{\sqrt{3}\over2}i &0\cr 0&k}, \
a_4=\pmat{k&0\cr0& -k},\
a_5=\pmat{j&0\cr0&j}.
\end{equation}
It is easily verified that 
$$ a_1^3=I, \quad a_2^2=a_3^2=a_4^2=a_5^2=-I, $$
and that
$$ a_1, \qquad a_1 a_2, \qquad a_2a_3, \qquad a_3a_4, $$
are reflections of order $3$, so that the groups $\langle a_1,\ldots,a_m\rangle $,
$1\le m\le 4$ are reflection groups, as is $\langle a_1,\ldots,a_5\rangle$. 
A simple calculation shows that groups $\langle a_1,a_2,a_3\rangle$,
$\langle a_1,a_2,a_3,a_4\rangle$,
$\langle a_1,a_2,a_3,a_4,a_5\rangle$ are exactly, i.e., 
have precisely the same elements, as the quaternionic reflection groups of
orders $120$, $720$, $1440$ having 
a primitive complexification given by the root systems $O_1,O_2,O_3$ of \cite{C80} (Table II).

The irreducible quaternionic reflection group $\langle a_1,a_2\rangle$ of order $24$ 
does not have a primitive complexification, and hence it can be viewed as a (primitive) 
complex reflection group in $U_2(\CC)$.
It is instructive to see how this happens, 
which leads to a final tweak of the generators to make this apparent.


\begin{example} Consider the irreducible quaternionic reflection group $G=\langle a_1,a_2\rangle$.
The action of the complexification of the generators on $\CC^4$
is given by 
$$ [a_1]_\CC=\pmat{ 1&0&0&0 \cr 0&-{1\over2}-{\sqrt{3}\over2}i&0&0 \cr
          0&0&1&0 \cr 0&0&0&-{1\over2}+{\sqrt{3}\over2}i}, \quad 
[a_2]_\CC = {i\over\sqrt{3}} \pmat{
        1 & 0 & 0 & \sqrt{2} \cr
        0 & 1 & \sqrt{2} & 0 \cr
        0 & \sqrt{2} & -1 & 0 \cr
        \sqrt{2} & 0 & 0 & -1 },  $$
and so we have  $\CC G$-invariant subspaces
$$ V_1=\spam_\CC\{e_1,e_4\}, \qquad V_2=\spam_\CC\{e_2,e_3\}, $$
of $\CC^4$, i.e., the complexification is not irreducible.
In other words, the $\CC G$-module $\HH^2$ (of dimension $4$) is not irreducible, as
it has $\CC G$-invariant subspaces
$$ W_1=\spam_\CC\{e_1,e_2 k\}, \qquad W_2=\spam_\CC\{e_1 i,e_2 j\}, $$
and can be written $\HH^2=W\oplus_\CC Wj$, where $W$ is either of these.
By changing the standard basis of the $\HH G$-module $\HH^2$ to one for 
a $\CC G$-invariant subspace, we obtain a representation where the 
matrices have complex entries. For example, take the basis $B=\{e_1 i,e_2 j\}$ for $W_1$, 
which has basis map
\begin{equation}
\label{udefn}
u=[e_1 ,e_2 k] = \pmat{1&0\cr0& k},
\end{equation}
to obtain
	$$ [a_1]_B=u^{-1} a_1 u = \pmat{1&0\cr0&\go}, \quad 
 [a_2]_B=u^{-1} a_2 u = \pmat{ {1\over\sqrt{3}}i & {\sqrt{2}\over\sqrt{3}} \cr
 -{\sqrt{2}\over\sqrt{3}} & -{1\over\sqrt{3}}i }, \qquad
	 \go:=-{1\over2}+{\sqrt{3}\over 2}i.	$$
\end{example}

In view of the above discussion, 
we take as generators for the Blichfeldt groups, the matrices
$$ b_j:=u^{-1}a_j u\in U_2(\HH), \qquad j=1,\ldots,5, $$
where $u$ is given by (\ref{udefn}), which are given by
\begin{equation}
\label{Blichaform}
b_1=\pmat{1&0\cr0&\go}, \
b_2=\pmat{ {1\over\sqrt{3}}i & {\sqrt{2}\over\sqrt{3}} \cr
 -{\sqrt{2}\over\sqrt{3}} & -{1\over\sqrt{3}}i }, \
b_3=\pmat{ -{\sqrt{3}\over2}i+ {1\over2}k &0\cr 0&k}, \
b_4=\pmat{k&0\cr0& -k},\
b_5=\pmat{j&0\cr0&j},
\end{equation}
and define a sequence of nested irreducible quaternionic reflection groups
\begin{equation}
\label{H720defn}
H_{24}=\langle b_1,b_2 \rangle, \quad
H_{120}=\langle b_1,b_2,b_3 \rangle, \quad
H_{720}=\langle b_1,b_2,b_3,b_4 \rangle, \quad
H_{1440}=\langle b_1,b_2,b_3,b_4,b_5 \rangle, 
\end{equation}
indexed by their orders. The first of these is the Shephard-Todd complex reflection 
group number $4$. We also let $H_3=\langle b_1 \rangle$, which is 
a reducible reflection group of order $3$.

\section{A nice presentation of the six equiangular lines}

Each of the six lines of (\ref{ETlines}) is fixed by a subgroup of $H$ of order $120$ (index $6$),
which is therefore reducible. The corresponding equiangular lines for $H_{720}$ can therefore 
be found as the orbit of a vector which is fixed by a reducible subgroup of order $120$.
The group $H_{720}$ has two subgroups of order $120$ up to conjugacy, which are 
isomorphic to $2\cdot A_5$ (the binary icosahedral group), with the class
length being six in both cases. One is $H_{120}$ which is an irreducible reflection group,
and the other is reducible and contains no reflections. 
By taking a reducible subgroup of order $120$, and finding a
line that it fixes (more detail later), one obtains the following ``fiducial'' vector
\begin{equation}
\label{wdefn}
w := \pmat{\sqrt{2}+\sqrt{10}\cr \sqrt{3}-i+j+\sqrt{3}k},
\end{equation}
whose orbit under $H_{720}$ is six equiangular lines.
With the ordering: 
$$ w, \quad b_1 w, \quad b_1^2 w, \quad b_2 w, \quad b_1b_2 w, \quad b_1^2 b_2 w, $$
they are those of (\ref{nicelines}), i.e., 
$$ \pmat{\sqrt{2}+\sqrt{10}\cr \sqrt{3}-i+j+\sqrt{3}k}, \
 \pmat{\sqrt{2}+\sqrt{10}\cr 2i-2j}, \
 \pmat{\sqrt{2}+\sqrt{10}\cr -\sqrt{3}-i+j-\sqrt{3}k}, $$
$$ \pmat{-\sqrt{3}-i-j-\sqrt{3}k \cr \sqrt{2}+\sqrt{10}}, \
\pmat{2i+2j \cr \sqrt{2}+\sqrt{10}}, \
\pmat{\sqrt{3}-i-j+\sqrt{3}k \cr \sqrt{2}+\sqrt{10}}.  $$

We observe that the six equiangular lines are an orbit of the Shephard-Todd complex reflection 
group $H_{24}$, and the absolute value of the ratio of coordinates is the golden ratio.  
With this ordering, Blichfeldt's generators correspond to the permutations
$$ b_1:\ (123)(456), \quad b_2:\ (14)(36), \quad b_3:\ (23)(45), \quad b_4:\ (13)(46). $$
In particular, $b_2$ fixes lines $2$ and $5$. 
The stabiliser in $H_{720}$ of the line given by the fiducial vector $w$ is generated by $b_3$ and an element
which can have order $3,5,6,10$, e.g., the following matrix of order $5$
\begin{equation}
\label{g2def}
g_2:=\pmat{ {1\over\sqrt{3}}i & {1\over 2\sqrt{6}}+{1\over 2\sqrt{2}}i +{1\over 2\sqrt{2}}j -{\sqrt{3}\over2\sqrt{2}}k \cr
   {1\over\sqrt{6}}+{1\over\sqrt{2}}j & -{1\over2}+ {1\over2\sqrt{3}} i }: \ 
   (24356).
\end{equation}

From the indexing of reflections by elements of $A_6$, it is immediate that $H_{720}$ cannot be 
generated by two reflections, but it can be by three, e.g., since we following action on the equiangular lines
$$ b_1:\ (123)(456), \quad b_1b_2:\ (156)(234),  \quad b_2b_3:\ (145)(263),
\quad b_3b_4:\ (123)(465), $$
it follows that $H_{720}$ is generated by the three reflections 
\begin{equation}
\label{nicerefgens}
b_1 = \pmat{1&0\cr0&-{1\over2}+{\sqrt{3}\over2}i}, \quad
b_1b_2 = \pmat{ {1\over\sqrt{3}}i & {\sqrt{2}\over\sqrt{3}} \cr {1\over\sqrt{6}} -{1\over\sqrt{2}i} & {1\over2}+{1\over2\sqrt{3}}i  }, \quad
b_3b_4 = \pmat{ -{1\over2}+{\sqrt{3}\over2}j & 0 \cr 0 & 1 }. 
\end{equation}

We now give an explicit conjugation 
$$ AHA^{-1}=H_{720}, \qquad A\in M_d(\HH), $$
which maps the lines $(v_j)$ of (\ref{ETlines}) to the lines $(w_j)$ of (\ref{nicelines}), i.e.,
$$ Av_j = w_j\ga_j, \qquad\exists \ga_j\in\HH. $$
If $\pm h_a$ and $\pm h_b$ are the elements of $H_{720}$ which give the permutations
$a$ and $b$ of (\ref{abpermdefs}) of the lines given by $(v_j)$, then a suitable $A$ is given by
one of each of the equations
$$ AU_aA^{-1}=\pm h_a, \quad AU_bA^{-1}=\pm h_b \Iff
AU_a=\pm h_a A, \quad AU_b=\pm h_b A.  $$
The latter presentation gives a homogeneous system of the linear equations in the 
entries of $A$, which we were able to solve (for a suitable choice of the $\pm$). 
In this way, we obtained the matrix
\begin{equation}
\label{Adefn}
A:= \pmat{ -{1\over2}-{\sqrt{5}\over2}
+\ga i
-\ga j
+({1\over2}+{\sqrt{5}\over2})k &
-2\gb+i-j+2\gb k \cr
-\gam 
+({\sqrt{2}\over2}+2-{\sqrt{2}\sqrt{5}\over2})i &
1-{3\sqrt{2}\over2}+\sqrt{5}-{\sqrt{2}\sqrt{5}\over2}-\gd i },
\end{equation}
$$ \mat{\ga= {\sqrt{2}\sqrt{3}\over3} + {\sqrt{2}\sqrt{3}\sqrt{5}\over3} -{5\sqrt{3}\over6}
-{\sqrt{3}\sqrt{5}\over6}, \qquad
\gb={\sqrt{2}\sqrt{3}\over3} -{\sqrt{3}\sqrt{5}\over6}} , $$
$$ \mat{\gam={\sqrt{2}\sqrt{3}\sqrt{5}\over6}+{\sqrt{2}\sqrt{3}\over2}-{2\sqrt{3}\over3}, \qquad
\gd= {\sqrt{2}\sqrt{3}\sqrt{5}\over6}+{\sqrt{3}\sqrt{5}\over3}-{\sqrt{2}\sqrt{3}\over6}+{\sqrt{3}\over3}}. $$
The diagonal entries of the scalar matrix $A^*A$ are 
$c={20\over3}(4-{\sqrt{5}\over5}(5\sqrt{2}-4)-\sqrt{2})>0$. Thus 
$U={1\over\sqrt{c}}A$ is a unitary matrix which gives the desired conjugation. 
However, the entries of $U$ are not in the cyclotomic field in which we did our calculations.

The conjugates of the $U_a$ and $U_b$ of (\ref{UaUbdefn}) give the following 
generators for $H_{720}$
\begin{align}
\label{UaUbconjugates}
AU_aA^{-1} &= \pmat{-{1\over2\sqrt{3}}i-{1\over\sqrt{3}}j+{1\over2}k & 
{1\over2\sqrt{6}}+{1\over2\sqrt{2}}i-{1\over2\sqrt{2}}j-{1\over2\sqrt{6}}k \cr
-{1\over2\sqrt{6}}+{1\over2\sqrt{2}}i-{1\over2\sqrt{2}}j-{1\over2\sqrt{6}}k & 
-{1\over\sqrt{3}}i-{1\over\sqrt{3}}j}, \cr
AU_bA^{-1} & = \pmat{ {1\over2}+{1\over\sqrt{3}}i-{1\over2\sqrt{3}}j &
{1\over2\sqrt{6}}+{1\over2\sqrt{2}}i+{1\over2\sqrt{2}}j+{1\over2\sqrt{6}}k \cr
-{1\over2\sqrt{6}}-{1\over2\sqrt{2}}i-{1\over2\sqrt{2}}j+{1\over2\sqrt{6}}k &
-{1\over2}+{1\over2\sqrt{3}}i+{1\over\sqrt{3}}j }.
\end{align}


\section{The stabiliser groups of the six equiangular lines}

We now consider the action of the stabiliser group of one of the six equiangular lines
on the line that it fixes. This action of a group of order $120$ on $\HH$ is far
from trivial 
(the situation for parabolic subgroups of reflection groups 
and highly symmetric tight frames \cite{BW13}), 
in fact it is faithful.
We first give some generalities about groups which fix a line, and what their 
action on the line can be.

Let $G\subset M_d(\HH)$ be a group, 
and consider its action on a nonzero vector $v\in\Hd$.
Two vectors $gv$ and $hv$ in the orbit give the same line if there is some
scalar $\ga\in\HH$ for which
$$ gv=hv\ga \Iff (h^{-1}g)v = v\ga. $$
Since such an $\ga$ is unique, we use the notation $\ga_g=\ga_{g,v}$ for the nonzero scalar with
\begin{equation}
\label{alphagdefn}
g v = v \ga_g.
\end{equation}
The {\bf stabiliser} (in $G$) of the line $L=\spam_\HH\{v\}$ 
given by a nonzero vector $v\in\Hd$, or 
the {\bf projective stabiliser} (in $G$) of the vector $v$, is defined by
$$ G_L= G_v := \{g\in G: gv =v\ga,\exists\ga\in\HH\}, $$
(this depends only on $L$), and the corresponding set of scalars are denoted by
$$ \HH_{G,v}^* = \{\ga_g\in\HH: gv=v\ga_g,g\in G_v\}. $$
In view of (\ref{alphagdefn}),
the matrix representation of $g|_L:L\to L$, the restriction of $g\in G_L$ to $L$, 
with respect to the $\HH$-basis $[v]$ for $L$ is
$ [g]:=[g|_L]_{[v]} = [\ga_g]$, and so 
$$ G_L\to\HH^*:g\mapsto \ga_g, \qquad G_L\to M_1(\HH):g\mapsto [\ga_g], $$
are group homomorphisms. Here $\HH^*$ denotes the group
of nonzero quaternions under multiplication. If $G_L$ is unitary, then the images above are 
unit scalars and unitary matrices, respectively.


\begin{proposition} Let $G\subset M_d(\HH)$ be a group,  
and $v\in\Hd$ a nonzero vector. 
Then the projective stabiliser $G_v$ is a subgroup of $G$, 
and $\HH_{G,v}^*$ is a subgroup of $\HH^*$ (being the
homomorphic image of $G_v\to\HH^*:g\mapsto\ga_g$), with
$$ G_v=G_{v\gb}, \qquad \HH_{G,v\gb}^* = \gb^{-1}\HH_{G,v}^*\gb, \qquad \gb\in\HH^*. $$
The projective stabilisers of points (lines) on the same $G$-orbit 
are conjugate and hence are isomorphic, i.e.,
$$ G_{hv}  = h^{-1} G_v h, \qquad h\in G. $$
\end{proposition}

\begin{proof}
We have already observed that $\HH_{G,v}^*$ is the isomorphic image of
$g\mapsto\ga_g$, or we can argue directly from 
(\ref{alphagdefn}) that
$$ v\ga_{gh}= (gh)v = g(hv) 
	= g(v\ga_h) 
	= (gv)\ga_h 
	= v\ga_g\ga_h 
\Implies \ga_{gh}=\ga_g\ga_h. $$
For any $\gb\in\HH^*$, we have
$$ \HH_{G,v\gb}=\{\ga: gv\gb = v\gb\ga,\exists\ga\}
	=\gb^{-1}\{\gb\ga\gb^{-1}: gv = v\gb\ga\gb^{-1},\exists\ga\}\gb
	= \gb^{-1}\HH_{G,v}^*\gb, $$
$$ G_{hv}=\{g:g(hv)=hv\ga,\exists\ga\}
	= h\{h^{-1}gh:h^{-1} ghv=v\ga,\exists\ga\}h^{-1}
=h G_v h^{-1}. $$
\end{proof}


If $G\subset M_d(\HH)$ is irreducible with $d\ge2$,
e.g., $G=H_{720}$,
then $G_v$ is reducible for $v$ nonzero, and so is a proper subgroup of $G$. 

\begin{example}
For $G=H_{720}$ and a $v$ giving one of the six
equiangular lines, the stabiliser subgroup $G_v$ has order $120$, and is
isomorphic to $2\cdot A_5$, the double cover of $A_5$. Since $A_5$ is simple,
$G_v$ has normal subgroups of orders $1$, $2$ and $120$. 
Since $\ga_{-I}=-1$, it follows (or by direct computation) 
that $\HH_{G,v}^*$, which is a quotient of $G_v$ by a normal subgroup,
	is $2\cdot A_5$. As a consequence, the orbit $(gv)_{g\in H_{720}}$ 
	has $720$ distinct vectors lying in six lines.
\end{example}

The faithful action of the stabiliser group of one of the six equiangular lines is
given by a subgroup $\HH_{G,v}^*$ of $\HH^*$ with order $120$. The finite 
subgroups of $\HH^*$, i.e., the reflection subgroups of $U_1(\HH)$,
have been classified by Stringham \cite{S81}
(also see \cite{C80}, \cite{CS03}). 

\begin{lemma}
\label{Stringhamlemma}
The finite subgroups of $\HH^*$, and hence of $U_1(\HH)$, up to conjugation, are
\begin{enumerate}[\rm(i)]
\item the cyclic group $\cC_m=\langle e^{2\pi i \over m}\rangle$, of order $m$, 
		$m\ge 1$,
\item the binary dihedral group $\cD_{m}=\langle \cC_{2m},k\rangle$, of order $4m$, $m\ge 2$,
\item the binary tetrahedral group $\cT=\langle \cD_2,{-1+i+j+k\over 2} \rangle$, 
	of order $24$, 
\item the binary octahedral group $\cO=\langle \cT,{i-1\over\sqrt{2}}\rangle$, of order $48$, 
\item the binary icosahedral group $\cI=\langle \cD_2, {\tau-\tau^{-1}i-j\over2}\rangle$,
	$\tau={1\over2}(1+\sqrt{5})$, of order $120$.
\end{enumerate}
None of the groups above are isomorphic, and the nontrivial ones are
the reflection groups in $U_1(\HH)$.
\end{lemma}

\begin{proof}
The list of conjugacy classes of finite subgroups is given in \cite{C80} as the above,
where $\cD_{m}$, has the index range $m\ge1$. It is clear that
$$ \cC_4=\langle i \rangle, \qquad \cD_1=\langle k\rangle, $$ 
are conjugate, and hence isomorphic. 
The $\cC_m$ and $\cD_m$ on the above list are abelian and nonabelian, respectively,
	and so are not isomorphic, nor are they isomorphic to the nonabelian groups 
	$\cT$, $\cO$, $\cI$
	(which contain $\cD_2$),
	i.e., $\cT\not\cong \cD_{6}$, $\cO\not\cong\cD_{12}$, $\cI\not\cong\cD_{30}$.
\end{proof}


From the above list, it follows that the stabiliser group $\HH_{G,v}^*$ for any one of the 
six equiangular lines is the binary icosahedral group $\cI\cong 2\cdot A_5$ 
(or a conjugate of it).

Let $H$ be the stabiliser group of the equiangular line given by the vector $w$ of (\ref{wdefn}), i.e.,
$$ H=(H_{720})_w = \langle b_3,g_2 \rangle, $$
where the generators $b_3$ and $g_2$ are given by (\ref{Blichaform}) and (\ref{g2def}).
We have observed that $W=\spam_\HH\{w\}$ is an irreducible $H$-submodule of $\HH^2$ on
which the action of $H$ is faithful. Since $H$ is unitary, it follows that the
orthogonal complement $W^\perp$ of $W$ is an irreducible $H$-submodule of $\HH^2$, 
and so its orbit gives a set of six lines. Let
\begin{equation}
\label{wperpdefn}
w^\perp := 
\pmat{ -\sqrt{3}-i+j+\sqrt{3}k \cr \sqrt{2}+\sqrt{10} } \in W^\perp.
\end{equation}
With the ordering: 
$w^\perp$, $b_1 w^\perp$, $b_1^2 w^\perp$, $b_2 w^\perp$, 
$b_1b_2 w^\perp$, $b_1^2 b_2 w^\perp$,
we have a second set of equiangular lines
\begin{align}
\label{wperpnicelines}
& \pmat{-\sqrt{3}-i+j+\sqrt{3}k \cr \sqrt{2}+\sqrt{10}}, \
\pmat{2i-2j \cr \sqrt{2}+\sqrt{10}}, \
\pmat{\sqrt{3}-i+j-\sqrt{3}k \cr \sqrt{2}+\sqrt{10}}, \cr
& \pmat{\sqrt{2}+\sqrt{10}\cr \sqrt{3}-i-j-\sqrt{3}k}, \
 \pmat{\sqrt{2}+\sqrt{10}\cr 2i+2j}, \
 \pmat{\sqrt{2}+\sqrt{10}\cr -\sqrt{3}-i-j+\sqrt{3}k},
\end{align}
which are an orbit of $H_{720}$. To understand the action of $H_{720}$ on the
line $W^\perp=\spam_\HH\{w^\perp\}$, we calculate the matrix representation for the 
basis $B=[w,w^\perp]$, i.e., $[g]_B=B^{-1} g B$, of the
generators $b_3$ (order $4$) and $g_2$ (order $5$) for $H$. They are
\begin{equation}
\label{b3g3diag}
[b_3]_B=\pmat{-{\sqrt{3}\over2}i+k & 0 \cr 0 & k}, \quad
[g_2]_B = \pmat{{\tau^{-1}\over2} + {\tau\over\sqrt{3}}i + {\tau^{-1}\over2\sqrt{3}}j & 0 \cr
0 & -{\tau\over2}+{\tau\over2\sqrt{3}}i-{\tau^{-1}\over2\sqrt{3}}j+{\tau^{-1}\over2}k},
\end{equation}
where $\tau={1\over2}(1+\sqrt{5})$. It is easily verified that
$$ \ga_{b_3,w^\perp} = -{\sqrt{3}\over2}i+k, \qquad 
\ga_{g_2,w^\perp} = {\tau^{-1}\over2} + {\tau\over\sqrt{3}}i + {\tau^{-1}\over2\sqrt{3}}j, $$
generate (a conjugate) of the binary icosahedral group $\cI$ (of Lemma \ref{Stringhamlemma}),
and so the action of $H$ on $W^\perp$ is faithful. 
However, $W$ and $W^\perp$ are not isomorphic $H$-submodules of $\HH^2=W\oplus W^\perp$,
since otherwise $\HH^2$ would be a homogeneous (isotypic) component for that
$H$-module, and hence the $H$-orbit of every vector in $\HH^2$ would be $1$-dimensional,
which is not the case.
In the interest of more general calculations, we show how this follows from 
character theory. 

Every quaternionic representation of a finite group $G$
as matrices in $M_d(\HH)$, such as $H$, corresponds to a complex representation
as matrices in $M_{2d}(\CC)$ via (\ref{symplecticform}).
The irreducible complex representations $\rho:G\to M_{2d}(\CC)$ that correspond to 
quaternionic representations are determined by the Frobenius–Schur indicator $\iota\chi$
of their character $\chi$, i.e.,
$$ \iota\chi:= {1\over|G|}\sum_{g\in g} \chi(g^2)\in\{-1,0,1\}, \qquad
\chi(g)=\trace(\rho(g)),$$
taking the value $\iota\chi=-1$ (see \cite{SS95}, \cite{G11}). For $H$ (as an abstract group) there are
two characters corresponding to quaternionic representations of rank $1$, i.e.,
$$ \mat{ 
\multicolumn{10}{c}{\hbox{\bf The rank $2$ characters of $H\cong 2\cdot A_5$}} \\[0.2cm]
\hbox{class size} & 1 & 1 & 20 & 30 & 12 & 12 & 20 & 12 & 12 \cr
\hbox{class order}& 1 & 2 &  3 &  4 &  5 &  5 &  6 & 10 & 10 \cr
\chi_1 & 2 & -2 & -1 & 0 & \tau^{-1} & -\tau  & 1 & \tau & -\tau^{-1} \cr
\chi_2 & 2 & -2 & -1 & 0 & -\tau & \tau^{-1} & 1 & -\tau^{-1} & \tau 
}
$$
There are also characters corresponding to irreducible quaternionic representations of
rank $2$ and rank $3$.
In view of (\ref{symplecticform}), 
the values of the character $\chi$ of the complexification of a 
quaternionic representation $G$ are given by
$$ \chi(g) = \trace([g]_\CC)=\trace(A)+\trace(\overline{A})=2\Re(\trace(g)). $$
Hence, by (\ref{b3g3diag}), the values of the characters of the representations 
of $H$ on $W$ and $W^\perp$ for the element $g_2$ are
$$ 2\Re(\trace(g_2|_W)
= 2\Re\Bigl({\tau^{-1}\over2} + {\tau\over\sqrt{3}}i + {\tau^{-1}\over2\sqrt{3}}j\Bigr)
= \tau^{-1} \qquad
2\Re(\trace(g_2|_{W^\perp})
= -\tau, $$
and so these representations are different.
We now summarise our calculations.

\begin{theorem} 
\label{twosetsequilines}
Let $G=H_{720}=\langle b_1,b_2,b_3,b_4\rangle\cong 2\cdot A_6$ 
be the primitive quaternionic reflection group of order $720$ 
given by (\ref{Blichaform}), 
with reducible subgroup $H=\langle b_3,g_2\rangle\cong 2\cdot A_5$
of order $120$, where $g_2$ is given by (\ref{g2def}),
and $w$ and $w^\perp$ be the orthogonal vectors 
\begin{equation}
\label{wandwperpdefn}
w = \pmat{\sqrt{2}+\sqrt{10}\cr \sqrt{3}-i+j+\sqrt{3}k}, \qquad
w^\perp = \pmat{ -\sqrt{3}-i+j+\sqrt{3}k \cr \sqrt{2}+\sqrt{10} }.
\end{equation}
Then the $G$-orbits of $w$ and $w^\perp$ each consist of $720$ distinct vectors 
which lie in set of six equiangular lines ($120$ vectors in each line), and 
$H$ fixes the  lines through $w$ and $w^\perp$, 
on which it therefore has a faithful irreducible action. Further, 
we have the orthogonal decomposition
$$ \HH^2 = \spam_\HH\{w\} \oplus \spam_\HH\{w^\perp\}, $$
of $\HH^2$ into non-isomorphic irreducible $H$-submodules,  i.e., 
	the homogeneous (isotypic) components.
\end{theorem} 

The fact that the orthogonal complement of equiangular lines
in $\HH^2$ gives another set of equiangular lines is an example of a more 
general phenomenon for $\HH^2$ ($d=2$).
We will refer to $|\inpro{v,w}|^2$ as the {\bf angle} between vectors $v,w\in\Hd$.

\begin{proposition}
For $v\in\HH^2$, let $v^\perp$ be any vector orthogonal to $v$, 
	with $\norm{v^\perp}=\norm{v}$, e.g.,
$$ v=\pmat{a\cr b}, \qquad 
v^\perp=\pmat{-\overline{a}^{-1}\overline{b}\overline{a}\cr \overline{a}}, 
\quad a\ne0. $$
Then ${}^\perp$ preserves the angles between lines, i.e.,
	$$ |\inpro{v,w}|^2=|\inpro{v^\perp,w^\perp}|^2, \qquad v,w\in\HH^2. $$
\end{proposition}

\begin{proof}
This is by direct computation. Suppose that $v=(a_1,b_1)$, $w=(a_2,b_2)$, with $a_1,a_2\ne0$
(the other cases being trivial). Then
$$ \inpro{v,w}=\overline{a_1}a_2+\overline{b_1}b_2, \qquad
\inpro{v^\perp,w^\perp} =
a_1 b_1 a_1^{-1} \overline{a_2}^{-1}\overline{b_2}\overline{a_2}
+a_1 \overline{a_2}. $$
Using the identities 
	$$ |a+b|^2=|a|^2+|b|^2+2\Re(a\overline{b}), \qquad \Re(ab)=\Re(ba), \qquad
	|a|^2a^{-1}=\overline{a}, \quad a\ne 0, $$
for $a,b\in\HH$, we calculate
\begin{align*}
|\inpro{v^\perp,w^\perp}|^2 &= |b_1|^2|b_2|^2+|a_1|^2|a_2|^2
+2\Re( a_1 b_1 a_1^{-1} \overline{a_2}^{-1}\overline{b_2}\overline{a_2}\cdot
a_2\overline{a_1}) \cr
& = |a_1|^2|a_2|^2+|b_1|^2|b_2|^2+2\Re(\overline{a_1} a_2\cdot \overline{b_2}b_1) \cr
	&=|\inpro{v,w}|^2,
\end{align*}
as claimed.
\end{proof}

\section{A general construction of interesting lines}

Our construction of the six equiangular lines involved the basic idea of 
taking the orbit of a vector/line which is fixed by some (ideally large) subgroup,
thereby giving
$$\hbox{ ``a small orbit with high symmetry''.} $$
In this way, one can obtain a finite class of ``lines with
high symmetry'', by an appropriate choice of definitions. For example:

\begin{itemize}{\it 
\item Let $G$ be a finite group with an irreducible action on $\Hd$.
\item Choose a maximal reducible subgroup $H\subset G$ (there are finitely many) 
	for which
$\Hd$ has at least one irreducible one-dimensional $H$-submodule $L=\spam_\HH\{v\}$
        of multiplicity $1$, and consider the $n=|G|/|H|$ lines $\{gL\}_{gH\in G/H}$.
}\end{itemize}
There are {\em finitely many} sets of ``highly symmetric lines'' which can be obtained in 
this way, from any given finite abstract group $G$. 
In our case $d=2$, and so every
reducible subgroup $H$ gives and a pair of orthogonal 
irreducible one-dimensional $H$-submodules.

If the action of $H$ on $L$ is trivial, i.e., $v$ is fixed, then one 
obtains what was called a ``highly symmetric tight frame''
in \cite{BW13} (these sets of $n$ vectors $\{gv\}_{gH\in G/H}$
were given for $\Cd$). When $G$ is a real, complex or quaternionic reflection group,
the corresponding subgroups $H$ giving a highly symmetric tight frame 
are said to be ``parabolic subgroups'', and these are known to be reflection
groups (\cite{S64}, \cite{BST23} for the quaternionic case).

If the action of $H$ on $L$ is not trivial, then one obtains the ``highly symmetric lines''
of \cite{G22} (given for $\Cd$, but also for $2$-dimensional subspaces, so 
this includes $\Hd$).

\begin{example} The irreducible reflection groups 
$H_{24}\subset H_{120}\subset H_{720}\subset H_{1440}$ contain reflections of order $3$,
with $H_{1440}$ also containing (and being generated by) reflections of order $2$.
A subgroup $H$ of one of these reflection groups has the 
trivial (identity) action on a $1$-dimensional subspace $L$ of $\HH^2$ 
if and only if each element of $H$ acts on $\HH^2$ as a reflection with root $L$.
Hence $H$ must be $H_3$ (up to conjugation), or a subgroup generated by 
	a reflection of order $2$ in $H_{1440}\setminus H_{720}$. 
These subgroups generated by a single reflection are therefore
the minimal, maximal and hence only nontrivial parabolic subgroups 
(see \cite{BST23} Lemma 4.2).
\end{example}

For $G=H_{720}$, there are $17$ reducible subgroups $H$ of order $>3$
        (up to conjugation), which therefore
have a nontrivial action on some line $L$ (and its orthogonal complement).
Three of these are maximal. 

\begin{example}
\label{maximal720}
The three maximal reducible subgroups of $G=H_{720}$ are
(as abstract groups) {\tt <120, 5>},  {\tt <48, 28>}, {\tt <36, 7>}, 
which correspond to sets of
$6$, $15$, $20$ lines, with angles $\{{2\over5}\}$, $\{{1\over4},{5\over8}\}$,
	$\{0,{1\over3},{2\over3}\}$, respectively. Generating (fiducial) vectors
for these sets of lines are
$$ \pmat{\sqrt{2}+\sqrt{10}\cr 2i-2j}, \qquad
\pmat{1\cr j}, \qquad
\pmat{1\cr0}. $$
The first is an equiangular vector from (\ref{nicelines}), and
	the third is a root vector (for $b_1$).
\end{example}


\noindent
These lines give optimal spherical designs, as we now explain.

For a fixed $t=1,2,\ldots$, a finite set unit vectors $\{v_1,\ldots,v_n\}$ in $\Hd$, 
equivalently lines, is said to be
a {\bf spherical $t$-design} (or {\bf spherical $(t,t)$-design}) if there is equality in
\begin{equation}
\label{Varchar}
\sum_{j=1}^n\sum_{k=1}^n |\inpro{v_j,v_k}|^{2t} \ge c_t(\Hd) \Bigl(
\sum_{\ell=1}^n \norm{v_\ell}^{2t}\Bigr)^2,
\qquad 
c_t(\Hd)
:= {2\cdot 3\cdots (t+1)\over 2d (2d+1)\cdots (2d+t-1)},
\end{equation}
i.e.,
\begin{equation}
\label{t-designdef}
{1\over n^2} \sum_{j=1}^n\sum_{k=1}^n |\inpro{v_j,v_k}|^{2t} = c_t(\Hd)
= {2\cdot 3\cdots (t+1)\over 2d (2d+1)\cdots (2d+t-1)}.
\end{equation}

The term ``spherical design'' comes from the fact that these can be viewed 
as cubature rules for the quaternionic sphere, as in the original presentation, 
which involved harmonic polynomials, see, e.g., \cite{H84}. It is quite technical 
(see \cite{W20c}) to show that this definition is equivalent to the 
``variational characterisation'' (\ref{t-designdef}), which is easier
to verify, and can be used to find numerical constructions.
A spherical $(1,1)$-design is a ``tight frame'' \cite{W20}, 
which is equivalent to the orthogonal-type expansion
\begin{equation} 
f=\sum_j v_j\inpro{v_j,f}, \qquad\forall f\in\Hd.
\end{equation}

Hoggar \cite{H79}, \cite{H82} 
gave (special) upper bounds on the number of unit vectors in $\Hd$ 
with a prescribed (small) angle set $A$, and (absolute) upper bounds which 
are independent of the angles $A$. 
For spherical $t$-designs for $\Rd$ and $\Cd$, there are lower bounds,
depending on $t$, which when they are (rarely) attained gives the
class of so called ``tight $t$-designs'' \cite{BMV04}, \cite{RS14}.
The corresponding theory for ``tight quaternionic spherical $t$-designs'' has yet to be 
developed (see \cite{CKM16}).

\begin{align*}
\multicolumn{5}{l}{\hbox{\bf The special bounds 
	and absolute bounds  
	for designs in $\Hd$}} \\[0.3cm]
	\hbox{angles $A$} & \qquad \hbox{special bound $\nu(A)$} && \qquad
	\hbox{absolute bound} \\[0.2cm]
	\{\ga\} & \qquad {d(1-\ga)\over 1-d\ga} &&\qquad d(2d-1) \\[0.2cm]
\{\ga,\gb\} & \qquad  {d(2d+1)(1-\ga)(1-\gb)\over3-(2d+1)(\ga+\gb)+d(2d+1)\ga\gb}
	&& \qquad {1\over3} d^2(4d^2-1) \\[0.2cm]
	\{0,\ga\} & \qquad {d(2d+1)(1-\ga)\over3-(2d+1)\ga} && \qquad {1\over3}d(4d^2-1)  \\[0.2cm]
        \{0,\ga,\gb\} & \qquad {d(d+1)(2d+1)(1-\ga)(1-\gb)\over6-3(d+1)(\ga+\gb)+(d+1)(2d+1)\ga\gb}   
	&& \qquad {1\over6}d^2 (d+1) (4d^2-1) \\[0.2cm]
\multicolumn{5}{l}{\hbox{\footnotesize
The $\nu(A)$ involving $\gb$ are subject to the restrictions
$\ga+\gb\le {3\over d+1}$ and $\ga+\gb\le {8\over 2d+3}$, respectively.
	}}
\end{align*}


The six equiangular lines $\HH^2$ at angle ${2\over5}$ satisfy both the 
special and absolute bounds. 
Before summarising the results of our calculations, we give some further details.

It is easy to determine whether the $n$ lines given by an orbit of $v$ 
are a spherical $t$-design by using (\ref{t-designdef}). Indeed, since
$$ \inpro{gv,hv}=\inpro{v,g^*hv}=\inpro{v,g^{-1}hv}, $$
if $(v_j)$ is a set of unit vectors giving the lines, then (\ref{t-designdef}) becomes
\begin{equation}
\label{t-designorb}
\sum_{j=1}^n |\inpro{v_j,w}|^{2t} = c_t(\Hd)\cdot n, 
\end{equation}
where $w$ is any vector in any line. It follows from this 
(with the multiplicities indicated) that the $6$ and $15$ lines 
with angles $\{{2\over5}\}$ and $\{{1\over4},{5\over8}\}$ are spherical 
$(2,2)$-designs, i.e.,
$$  1+5\cdot \Bigl({2\over5}\Bigr)^2 ={3\over 10}\cdot  6, \qquad
1+6\cdot\Bigl({1\over4}\Bigr)^2+8\cdot\Bigl({5\over8}\Bigr)^2
= {3\over10}\cdot 15. $$
Moreover, the $20$ lines with angles $\{0,{1\over3},{2\over3}\}$ 
give 
a spherical $(3,3)$-design, since
$$ 1+1\cdot0+9\cdot\Bigl({1\over3}\Bigr)^3+9\cdot\Bigl({2\over3}\Bigr)^3
= {1\over5}\cdot 20. $$

It was shown in \cite{MW19} that higher order real and complex spherical designs could be
obtained by taking a union of orbits. 
This concept extends to 
quaternionic designs:

\begin{example} 
\label{MUBequiangular}
(Mutually unbiased equiangular lines)
Take the union of the two sets of six equiangular lines in $\HH^2$ at angle 
${2\over5}$ given in Theorem \ref{twosetsequilines}, 
i.e., the orbit of the orthogonal vectors $w$ and $w^\perp$.
It is easily verified that the angle between any vector in one set of lines and
and any of the five from the other set which are not orthogonal to it is ${3\over5}$. 
It follows from  (\ref{t-designorb})
that these $12$ vectors with angles $\{0,{2\over3},{3\over5}\}$
give a spherical $(3,3)$-design for $\HH^2$, by the calculation
$$  1+1\cdot 0+5\cdot\Bigl({2\over5}\Bigr)^3 + 5\cdot\Bigl({3\over5}\Bigr)^3
= {1\over 5} \cdot 12. $$
\end{example}

The above union of two sets of equiangular lines can be viewed as an orbit of $H_{1440}$,
i.e., as highly symmetric lines for $H_{1440}$.

\begin{example}
\label{H1440example}
The reflection group $G=H_{1440}$ has five maximal reducible 
subgroups of orders $120,72,48,48,24$ corresponding to sets of 
$12,20,30,30,60$ highly symmetric lines.
Since the index of $H_{720}$ in $G$ is $2$, the $G$-orbit of a set of $n$ lines 
that are an $H_{720}$-orbit is either the same set of lines or a set of $2n$ lines.
In this way, for the $6,15,20$ lines of Example \ref{maximal720}, we obtain
sets of $12,30,20$ lines. 


The $12$ lines are those of Example \ref{MUBequiangular}.
Indeed, if $r\in H_{1440}$ is the reflection of order $2$ given by
$$ r:={1\over\sqrt{2}}\pmat{0 & 1+k \cr 1-k & 0}, $$
then it maps the equiangular lines given by $w$ and $w^\perp$ of (\ref{wandwperpdefn}) to
each other, in particular 
$$ r w = w^\perp\ga, \qquad\ga= {1-k\over\sqrt{2}}. $$
Further, $H_{1440}$ maps the six line orthogonal line pairs $v\HH\cup v^\perp\HH$ 
(crosses) to
each other. This gives a permutation representation of $H_{1440}$, with kernel $\inpro{-I}$.
Hence $H_{1440}$ is $2\cdot S_6$, the double cover of $S_6$, 
and elements $\pm g$ in $H_{1440}$ can be indexed by permutations on the six equiangular lines
(with even permutations mapping a given set of equiangular lines to itself, and odd permutations
	mapping it to the set of orthogonal equiangular lines).

The $30$ lines obtained from the $15$ lines with angles $\{{1\over4},{5\over8}\}$,
which is a $(2,2)$-design, give a spherical $(3,3)$-design
with angles $\{0,{1\over4},{3\over8},{5\over8},{3\over4}\}$, via the calculation
\begin{equation}
\label{thirtylinesI}
1+1\cdot 0+6\cdot\Bigl({1\over4}\Bigr)^3 + 8\cdot\Bigl({3\over8}\Bigr)^3 
+ 8\cdot\Bigl({5\over8}\Bigr)^3 + 6\cdot\Bigl({3\over4}\Bigr)^3
= {1\over 5} \cdot 30.
\end{equation}
\end{example}

The list of \cite{H82} gives one $t$-design in $\HH^2$ meeting the special (and also
absolute) bound
i.e., 
the $10$ vectors given by the five
mutually 
unbiased orthonormal bases
\begin{equation}
\label{H2MUBs}
(1,0),\ (0,1),\ (1,\pm1),\ (1,\pm i),\ (1,\pm j),\ (1,\pm k), 
\end{equation}
with angles $\{0,{1\over2}\}$, which form a spherical $(3,3)$-design.
 
Our calculations have given three new spherical designs for $\HH^2$ that meet the special 
bound, in addition to the two others known.

\begin{align*}
\multicolumn{11}{l}{\hbox{\bf
The special and absolute bounds for spherical $t$-designs in $\HH^2$ }} \\[0.2cm]
	n & \qquad A &&\quad \nu(A) &&\quad \hbox{\footnotesize\vbox{\hbox{absolute}\hbox{bound}}}
	&&\ t &&\qquad  \\[0.2cm]
	6 & \qquad \hbox{$\{{2\over5}\}$} &&\quad  6 &&\quad 6 &&\ 2 &&\qquad \hbox{Equiangular lines \cite{ET20}} \\[0.2cm]
	10 & \qquad \hbox{$\{0,{1\over2}\}$} &&\quad 10 &&\quad 10 &&\ 3 &&\qquad \hbox{\cite{H82} (Example 3)} \\[0.2cm]
	12 & \qquad \hbox{$\{0,{2\over5},{3\over5}\}$} &&\quad 12 &&\quad 30 &&\ 3 &&\qquad \hbox{Example \ref{MUBequiangular}, Example \ref{H1440example}
 }\\[0.2cm]
15 & \qquad \hbox{$\{{1\over4},{5\over8}\}$} &&\quad 15 &&\quad 20 &&\ 2 &&\qquad \hbox{Example \ref{maximal720} }\\[0.2cm]
20 & \qquad \hbox{$\{0,{1\over3},{2\over3}\}$} &&\quad 20 &&\quad 30 &&\ 3 &&\qquad \hbox{Example \ref{maximal720} } 
\end{align*}

\vskip0.5truecm

Interesting designs can also be found as highly symmetric lines for 
nonmaximal reducible subgroups of $G$, 
e.g., the six equiangular lines for $G=H_{1440}$. Here is another.

\begin{example} Let $G=H_{720}$. This has a nonmaximal reducible subgroup of order $24$ 
(a subgroup of the maximal reducible subgroups of orders $120$ and $48$),
which gives a system of $30$ lines with angles $\{0,{1\over4},{1\over2},{3\over4}\}$ 
generated by the fiducial vector
	$$ \pmat{\sqrt{2}i\cr 1+\sqrt{3}}. $$
This gives a spherical $(3,3)$-design, via
$$ 1+1\cdot 0 +8\cdot\Bigl({1\over4}\Bigr)^3 +12\cdot\Bigl({1\over2}\Bigr)^3 
+8\cdot\Bigl({3\over4}\Bigr)^3 = {1\over 5} \cdot 30. $$
This design is fixed by $H_{1440}$, and so, in view of (\ref{thirtylinesI}), 
it gives the second set of $30$ lines mentioned in Example \ref{H1440example}.
\end{example}

We now give details about how the calculations discussed were implemented.

\section{Computational details}

The calculation of all subgroups of a given group $G$ (up to conjugacy) is
a task easily done in Magma using {\tt Subgroups(G)}. The identification
of those subgroups that are reducible groups of quaternionic matrices, and hence give
systems of lines is a little more involved. Serre's condition for irreducibility
of finite groups of real or complex matrices \cite{S77} (Theorem 5, Chapter 2)
\begin{equation}
\label{Serrecdn}
\sum_{g\in G} \trace(g)\trace(g^{-1}) =|G|,
\end{equation}
cannot be applied, or generalised in a routine way. 
However, it was shown in \cite{W20}
that a finite group $G\subset U_d(\HH)$ is irreducible if and only if 
every orbit of a nonzero vector is a ``tight frame'', i.e., is a spherical 
$(1,1)$-design. From (\ref{t-designdef}), it follows that the orbit of a 
nonzero vector $x\in\Hd$ is spherical $(t,t)$-design if and only if
\begin{equation}
\label{pG(t)defn}
p_G^{(t)}(x) := {1\over|G|}\sum_{g\in G}|\inpro{x,gx}|^{2t}-c_t(\Hd) \inpro{x,x}^{2t} = 0.
\end{equation}
In particular, for $t=1$, we have the condition for being irreducible
\begin{equation}
\label{irreduciblecdn}
{1\over|G|}\sum_{g\in G}|\inpro{x,gx}|^{2}-{1\over d}\inpro{x,x}^{2} = 0.
\end{equation}
This is easily verified in Magma by using {\tt PolynomialRing} to set up an appropriate 
polynomial ring with the coordinates of $x$ as variables.
In this way, we calculated
$$ p_{H_{24}}^{(1)}=0, \qquad p_{H_{720}}^{(2)}=0, \qquad p_{H_{1440}}^{(3)}=0, $$
which gives
\begin{itemize}
\item Every $H_{24}$ orbit of a nonzero vector gives a spherical $(1,1)$-design.
\item Every $H_{720}$ orbit of a nonzero vector gives a spherical $(2,2)$-design.
\item Every $H_{1440}$ orbit of a nonzero vector gives a spherical $(3,3)$-design.
\end{itemize}

Given a reducible subgroup $G\subset U_d(\HH)$, our method requires the calculation of any 
$1$-dimensional $G$-invariant subspaces $x\HH$ that may exist, 
i.e., those nonzero $x\in\Hd$ for which 
\begin{equation}
\label{1dimGinvarsubspace}
gx=x\ga_g, \quad \exists\ga_g\in\HH, \qquad\forall g\in G. 
\end{equation}
The Cauchy-Schwarz inequality (and equality) extends to $\Hd$ (see \cite{W20}), so
that (\ref{1dimGinvarsubspace}) holds if and only if there is the equality
\begin{equation}
\label{CSequivcdn}
|\inpro{gx,x}|^2=\inpro{gx,gx}\inpro{x,x}=\inpro{x,x}^2,
\end{equation}
and hence the set of $x\in\Hd$ giving $1$-dimensional $G$-invariant subspaces 
is an algebraic variety.

\begin{lemma}
Let $G\subset U(\Hd)$ be reducible and $x\in\Hd$ be nonzero. 
Then the line $x\HH$ is 
fixed by $G$ if and only if
\begin{equation}
\label{fixedlineeqn}
|\inpro{gx,x}|^2 = \inpro{x,x}^2, \qquad \forall g\in\cG, 
\end{equation}
where $\cG$ is any generating set for $G$.
\end{lemma}

\begin{proof} Use the condition (\ref{CSequivcdn}), 
and the observation that $G$-invariance is equivalent to 
invariance under a generating set for $G$.
\end{proof}

We denote the real algebraic variety given by the set of solutions $x\in\Hd$ to the 
system of polynomial equations (\ref{fixedlineeqn}) by $\cV_1(G)$.
For us, the computation of $\cV_1(G)$ was not completely straightforward,
for the following reasons:

\begin{itemize}
\item The system of $|\cG|$ polynomial equations (\ref{fixedlineeqn}) in the $4d$ real variables given by the
$1,i,j,k$ coefficients of the coordinates of $x\in\Hd$ is easily formed in Magma. However,
in many cases, it could not be solved there. In these cases, the software system Maple 
was then used to find a numerical solutions, from which analytic ones could then be deduced.
\item For reducible subgroups $H_1\subset H_2$, we have that $\cV_1(H_2)\subset\cV_1(H_1)$.
Often, when finding an element of $\V_1(H_1)$ for a nonmaximal reducible subgroup $H_1$
of $G=H_{720}$, it turned out to be in the algebraic variety for a maximal subgroup
(Example \ref{maximal720}). Given that the whole variety was not being calculated, it was
hard to form a clear picture of whether this was some quirk of the calculation method
		or because $\cV_1(H_1)\setminus \cV_1(H_2)$ might be empty. 
		This is an ongoing investigation.
\end{itemize}

\section{Concluding remarks}

We have shown that the unique maximal set of six equiangular lines in $\HH^2$
is the orbit of a (quaternionic) reflection group. The same is also true for the maximal
sets of equiangular lines in $\RR^2$ and $\CC^2$. The three equiangular lines in $\RR^2$
(the Mercedes-Benz frame) are an orbit of the faithful irreducible action of $S_3$ on $\RR^2$,
and the SIC of four equiangular lines in $\CC^2$ is an orbit of the complex reflection groups 
with Shephard-Todd numbers $4,5,6,7$ (see \cite{BW13}, Table 1). 
Indeed, the first of these groups is $H_{24}$, and the $H_{24}$-orbit of the vector $e_1=(1,0)$
gives the following presentation of the SIC
$$ \pmat{1\cr 0}, \
\pmat{{1\over\sqrt{3}}i\cr-{\sqrt{2}\over\sqrt{3}}}, \
\pmat{{1\over\sqrt{3}}i\cr{1\over\sqrt{6}}-{1\over\sqrt{2}}i},\
\pmat{{1\over\sqrt{3}}i\cr{1\over\sqrt{6}}+{1\over\sqrt{2}}i}. $$
This SIC is also the orbit of the discrete Heisenberg group \cite{W18}, which
(in this case  $d=2$) is an irreducible real reflection group of order $8$.

Moreover, the Hesse SIC of nine equiangular lines in $\CC^3$ 
is an orbit of the
complex reflection groups with Shephard-Todd numbers $25,26$ (\cite{BW13}, Example 11,
Table 2).

The above examples notwithstanding, we do not expect that the maximal sets of 
quaternionic equiangular lines in $\Hd$ come as orbits of quaternionic reflection groups, 
in general, for the following reasons:

\begin{itemize}
\item By considering eigenspaces, any element of a finite group $G\subset U_2(\CC)$ 
can be multiplied by a unit scalar, to obtain a reflection of the same order.
In this way, the appearance of reflection groups 
for collineation groups acting on $\CC^2$ is incidental, 
rather than by design. Similar reasoning can be applied to collineation groups acting on $\HH^2$.
\item The general method for finding maximal sets of equiangular lines in $\Cd$ is as the orbit of a fiducial vector
under the action of the Heisenberg group (an irreducible projective representation of $\ZZ_d^2$). 
These can be viewed as highly symmetric lines for a larger Clifford group \cite{W18}. 
For $d\ge3$, the Heisenberg and Clifford groups are not given by reflection groups.
For $\Hd$, Cohen's classification \cite{C80} gives no infinite families of 
irreducible quaternionic reflection groups. 
Thus, to find infinite families of optimal quaternionic equiangular lines in this way, 
one would like an infinite family of irreducible quaternionic matrix groups, of which 
			none are currently known.
\end{itemize}

We conclude with a couple of obvious directions for extending number of known maximal
sets of quaternionic lines:

\begin{itemize} 
\item Starting with a given quaternionic reflection group, find the associated
		sets of highly symmetric lines -- hoping for an equiangular set.
\item Use the variational characterisation of (\ref{Varchar}), or other methods, to find 
sets of quaternionic equiangular lines numerically (see \cite{CKM16}), and 
then deduce the irreducible projective quaternionic representation that they
might be given by.
\end{itemize}

\section{Acknowledgements}

Thanks to Gabriel Verret, Don Taylor and Allan Steel for some very useful discussions about
the implementation of quaternionic reflection groups in Magma.

\vfil\eject

\bibliographystyle{alpha}
\bibliography{refs-sixlines}
\nocite{*}

\end{document}